\documentclass[oneside]{amsart}
\usepackage{preamble}
\title{Hyperfiniteness for group actions on trees}

\author[Srivatsav Kunnawalkam Elayavalli]{Srivatsav Kunnawalkam Elayavalli}
\address{Department of Mathematical Sciences, UCSD, 9500 Gilman Dr, La Jolla, CA 92092, USA}
\email{skunnawalkamelayaval@ucsd.edu}
\urladdr{https://sites.google.com/view/srivatsavke}

\author{Koichi Oyakawa}
\address{Department of Mathematics\\Vanderbilt University, 1326 Stevenson Center, Station B 407807, Nashville, TN 37240, USA}
\email{koichi.oyakawa@vanderbilt.edu}
\urladdr{https://sites.google.com/view/koichi-oyakawa}

\author{Forte Shinko} 
\address{University of California, Berkeley, Department of Mathematics, 970 Evans Hall, Berkeley, CA 94720, USA}
\email{forteshinko@berkeley.edu}
\urladdr{https://math.berkeley.edu/~forte/}

\author{Pieter Spaas}
\address{Department of Mathematical Sciences, University of Copenhagen, Universitetsparken 5, DK-2100 Copenhagen \O, Denmark}
\email{pisp@math.ku.dk}
\urladdr{https://sites.google.com/view/pieterspaas}

{\thanks{S.K.E was supported by a Simons Postdoctoral Fellowship. 
P.S. was partially supported by a research grant from the Danish Council for Independent Research, Natural Sciences, and partially by MSCA Fellowship No. 101111079 from the European Union. }}

\begin{document}

\begin{abstract}
    We identify natural conditions for a countable group acting on a countable tree which imply that the orbit equivalence relation of the induced action on the Gromov boundary is Borel hyperfinite. Examples of this condition include acylindrical actions. We also identify a natural weakening of the aforementioned conditions that implies measure hyperfiniteness of the boundary action. We then document examples of group actions on trees whose boundary action is not hyperfinite.
\end{abstract}

\maketitle
\section{Introduction}

Recently there has been a trend of proving Borel hyperfiniteness for
orbit equivalence relations associated to natural actions of countable groups.
This project sits at the intersection between descriptive set theory,
ergodic theory and geometric group theory.
While it is a major open problem if all actions of amenable groups
give rise to hyperfinite equivalence relations
(see \cite{CJMST23} for partial results and further references),
the problem of identifying hyperfiniteness for natural actions
associated to non-amenable groups is also of great interest
and is seen to be tractable in certain cases. 

Presently we are inspired by the following results recorded in chronological order,
that summarize the activity in the aforementioned problem.
First Dougherty, Jackson and Kechris in \cite{DJK94} showed that
the tail equivalence relation is hyperfinite,
which implies that the action of $\mathbb{F}_2$ on the Gromov boundary
of its Cayley graph gives a Borel hyperfinite equivalence relation.
Then Huang, Sabok and Shinko \cite{HSS20}
generalized the above result to certain actions of hyperbolic groups.
This was pushed to all hyperbolic groups in \cite{MS20},
with a new proof given recently in \cite{NV23}.
Other examples include certain actions of mapping class groups \cite{PS21},
relatively hyperbolic groups acting naturally on the Bowditch boundary \cite{Kar22},
and in the general setting of acylindrically hyperbolic groups \cite{Oya24}. 

The first main result of this note is the following:
\begin{mthm}\label{mthm:A}
    Let $G \car T$ be an action of a countable group on a countable tree,
    and suppose that every geodesic ray $\mathbf v$ in $T$
    has an initial segment $\sigma$ with
    $\Stab(\sigma) = \Stab(\mathbf v)$.
    Then the induced action of $G$ on the Gromov boundary $\partial T$
    is Borel hyperfinite.
\end{mthm}

Note that the tree need not be locally finite and the action need not be free.
An example of this is when the action is acylindrical (see \cite{Sel97, Osi18}),
a very well studied notion in geometric group theory.
Indeed acylindricity implies that every geodesic ray
has an initial segment with finite stabilizer.
The proof of the above main result relies on reducing the problem to tail equivalence,
which was shown to be hyperfinite in \cite{DJK94}.

An important particular case of acylindrical actions on trees
is the action of the fundamental group of a closed orientable 3-manifold $M$
on the Bass-Serre tree associated with the JSJ decomposition of $M$
(see Lemma 2.4 in \cite{WZ10}).
If $M$ is not virtually a torus bundle over a circle,
our above result implies hyperfiniteness;
otherwise, $\pi_1(M)$ is virtually polycyclic and hyperfiniteness
follows from \cite[Corollary 1.9]{CJMST23}.
Thus we get the following
(this application of \Cref{mthm:A} was pointed out by Denis Osin):

\begin{cor}\label{osin}
    Suppose that $M$ is a closed, orientable, irreducible $3$-manifold.
    Then the action of $\pi_1(M)$ on the Bass-Serre tree $T$
    associated with the JSJ decomposition of $M$
    induces a Borel hyperfinite equivalence relation on the Gromov boundary of $T$.
\end{cor}

We also identify a natural weakening of the condition in \Cref{mthm:A} for the action
so that the induced action on the boundary is measure-hyperfinite,
in fact Borel 2-amenable.
The following is the next main result:

\begin{mthm}\label{mthm:B}
    Let $G \car T$ be an action of a countable group on a countable tree
    such that every geodesic ray $\mathbf v$
    has an initial segment $\sigma$ such that
    $\Stab(\mathbf v)$ is uniformly coamenable
    (see \Cref{uca}) in $\Stab(\sigma)$.
    Then the induced action of $G$ on the Gromov boundary $\partial T$
    is Borel $2$-amenable,
    so in particular measure-hyperfinite.
\end{mthm}

A question that arises naturally from the above results is the following:
Do there exist countable group actions on trees for which
the induced equivalence relation on the Gromov boundary is \emph{not} Borel hyperfinite.
We answer this question affirmatively and collect some remarks in \Cref{prodiscrete}.
In particular, we identify natural such examples arising from
actions on trees of certain inverse limits of countable groups.

\subsection*{Acknowledgements}
We thank Kyoto University for hosting three of the authors for the 8th KTGU Mathematics Workshop for Young Researchers,
where the idea for this note was conceived.
We thank Adrian Ioana, Andrew Marks and Damian Osajda for helpful comments.
We thank Denis Osin for suggesting \Cref{osin}.

\section{Preliminaries}
\subsection{Equivalence relations}
Let $E$ be an equivalence relation on a set $X$,
and let $A \subseteq X$.
We say that $A$ is a \textbf{transversal} for $E$
if $A$ meets every $E$-class exactly once.
We say that $A$ is \textbf{$E$-invariant}
if it is a union of $E$-classes.
The \textbf{saturation} of $A$,
denoted $[A]_E$,
is the smallest $E$-invariant set containing $A$.
We write $E \uhr A$ for the restriction of $E$ to $A$.

Given a monoid action $M \car X$,
its \textbf{orbit equivalence relation},
denoted $E_M^X$,
is the equivalence relation on $X$
generated by the relation
$\{(x, mx) \in X^2 : x \in X, m \in M\}$.
We write $X/M$ to mean $X/E_M^X$.

A \textbf{countable Borel equivalence relation (CBER)}
is an equivalence relation $E$ on a standard Borel space $X$,
such that $E$ is Borel as a subset of $X^2$,
and such that every $E$-class is countable.

A CBER $E$ on a standard Borel space $X$ is \textbf{smooth}
if the following equivalent conditions hold
(see \cite[Proposition 3.12]{Kec24}):
\begin{enumerate}
    \item $E$ has a Borel transversal.
    \item The quotient space $X/E$ is standard Borel.
\end{enumerate}

Let $E$ and $F$ be CBERs on standard Borel spaces $X$ and $Y$ respectively.
We say that $E$ \textbf{Borel reduces} to $F$,
written $E \le_B F$,
if there is a Borel map $f : X \to Y$
such that for all $x, x' \in X$,
we have $x \mr E x' \iff f(x) \mr F f(x')$.
Note that Borel reduction is a preorder on CBERs.
We say that $E$ is \textbf{Borel bireducible} with $F$,
written $E \sim_B F$,
if the following equivalent conditions hold
(see \cite[Theorem 3.32]{Kec24}):
\begin{enumerate}
    \item $E \le_B F$ and $F \le_B E$.
    \item There is a Borel map $X \to Y$
        which descends to a bijection $X/E \to Y/F$.
\end{enumerate}

We will make use of the following observation.
A monoid action $M \car X$ is \textbf{countable-to-one}
if for every $x \in X$,
there are only countably many $(m, x') \in M \times X$
with $mx' = x$.

\begin{lem}\label{quotient}
    Let $M$ and $N$ be countable monoids,
    and let $M \times N \car X$ be a countable-to-one
    Borel action on a standard Borel space
    such that $E_N^X$ is smooth.
    Then $E_{M \times N}^X \sim_B E_M^{X/N}$.
\end{lem}

\begin{proof}
    The quotient map $X \thra X/N$
    descends to a bijection $X/(M \times N) \hthra (X/N)/M$.
\end{proof}

A CBER $E$ is \textbf{Borel hyperfinite}
if $E = \bigcup_i E_i$
for some increasing sequence $(E_i)_{i \in \N}$ of CBERs,
each of which has all classes finite.

\subsection{Trees}
A \textbf{tree} is a connected acyclic undirected simple graph.

Let $T$ be a countable tree.
A \textbf{geodesic} (resp. \textbf{geodesic ray}) in $T$
is a finite (resp. infinite)
injective sequence of vertices of $T$
such that subsequent terms are adjacent in $T$.
Given a geodesic $\sigma$ and a geodesic ray $\mbf v$,
we write $\sigma \prec \mbf v$
if $\sigma$ is a strict initial segment of $\mbf v$
(so $\sigma$ is necessarily finite).
Let $\Geo(T)$ denote the standard Borel space of geodesic rays in $T$,
viewed as a Borel subset of $V(T)^\N$, where $V(T)$ is the vertex set of $T$.
There is an action $\N \car \Geo(T)$ defined by
$1 \cdot (v_0, v_1, v_2, \ldots)
= (v_1, v_2, v_3, \ldots)$.
The quotient space $\Geo(T)/\N$
is called the \textbf{boundary} of $T$,
and is denoted $\partial T$.
This is a standard Borel space since $E_\N^{\Geo(T)}$ is smooth:
for a fixed vertex $v \in T$,
the set of geodesic rays starting at $v$ is a Borel transversal.
The action $\Aut(T) \car T$ commutes with the action $\N \car T$,
so it descends to a Borel action $\Aut(T) \car \partial T$.

\section{Borel hyperfinite boundary actions}
\begin{lem}\label{smoothUnion}
    Let $E$ be a CBER on a standard Borel space $X$,
    and let $(A_n)_{n \in \N}$ be a cover of $X$ by Borel sets.
    If each $E \uhr A_n$ is smooth,
    then so is $E$.
\end{lem}
\begin{proof}
    Each $E \uhr [A_n]_E$ is smooth,
    since any Borel transversal for $E \uhr A_n$
    is also a Borel transversal for $E \uhr [A_n]_E$.
    So by replacing each $A_n$ with its saturation,
    we can assume that each $A_n$ is $E$-invariant.
    By passing to subsets,
    we can further assume that $(A_n)_{n \in \N}$ is a partition.
    Then each $E \uhr A_n$ has a Borel transversal,
    and the union of these transversals is a Borel transversal for $E$.
\end{proof}

We can now prove \Cref{mthm:A} from the introduction.

\begin{thm}\label{thm:BorelHyp}
    Let $G \car T$ be an action of a countable group on a countable tree
    such that every geodesic ray $\mathbf v$ in $T$
    has an initial segment $\sigma$ with
    $\Stab(\mathbf v) = \Stab(\sigma)$.
    Then $E_G^{\partial T}$ is Borel hyperfinite.
\end{thm}
\begin{proof}
    For every geodesic $\sigma$,
    let $\Geo_\sigma(T)$ be the set of geodesic rays $\mathbf v$ extending $\sigma$
    such that $\Stab(\mathbf v) = \Stab(\sigma)$.
    Then $(\Geo_\sigma(T))_\sigma$ is a cover of $\Geo(T)$ by our condition and for any geodesic $\sigma$, 
    $\Geo_\sigma(T)$ is closed.
    Furthermore, for every $\sigma$,
    we have
    $E_G^{\Geo(T)} \uhr \Geo_\sigma(T)
    = E_{\Stab(\sigma)}^{\Geo_\sigma(T)}$.
    Now we see that the action $\Stab(\sigma) \car \Geo_\sigma(T)$ is trivial, since $\Stab(\mathbf v) = \Stab(\sigma)$ for any $\mathbf v \in \Geo_\sigma(T)$,
    hence this equivalence relation is smooth.
    We conclude that $E_G^{\Geo(T)}$ is smooth by \Cref{smoothUnion}, 
    and thus $E_\N^{\Geo(T)/G}$ is Borel hyperfinite by \cite[Corollary 8.2]{DJK94}.
    By applying \Cref{quotient} twice,
    we have $E_G^{\partial T} \sim_B E_{G \times \N}^{\Geo(T)} \sim_B E_\N^{\Geo(T)/G}$.
    Since hyperfiniteness is closed under Borel reduction
    (see \cite[Proposition 1.3(ii)]{JKL02}),
    we get that $E_G^{\partial T}$ is Borel hyperfinite.
\end{proof}

\begin{eg}
    A natural example of an action satisfying the condition from \Cref{thm:BorelHyp} is any amalgamated free product $G = H *_C K$
    acting on its Bass-Serre tree (\cite{Ser80}),
    when the amalgam $C$ is almost malnormal.
    Indeed,
    in that case $gCg^{-1}\cap C$ is finite for every $g\in G\setminus C$ by definition,
    which means that the stabilizer of any segment of length at least $2$ is finite.
\end{eg}

\section{Measure-hyperfinite boundary actions}
Let $A$ be a set.
View $A$ as a subset of $\ell^1(A)$,
by treating elements of $\ell^1(A)$
as formal $\R$-linear combinations of elements of $A$.
A \textbf{probability measure} on $A$
is a non-negative element $p \in \ell^1(A)$ with $\|p\|_1 = 1$.
Let $\Prob(A)$ denote the set of probability measures on $A$.
Any function $A \times B \to C$ extends linearly
to a function $\Prob(A) \times \Prob(B) \to \Prob(C)$.
In particular,
if $M$ is a monoid,
then $\Prob(M)$ is a monoid,
and every monoid action $M \car X$ extends
to an action $\Prob(M) \car \Prob(X)$.

A CBER $E$ on $X$ is \textbf{Borel amenable}
if there is a sequence $(p_i)_{i \in \N}$,
where each $p_i$ is a function taking every $x \in X$
to some $p_i^x \in \Prob([x]_E)$,
such that
\[
    \forall (x, y) \in E\; 
    \forall \ve > 0 \;
    \forall^\infty i \;
    \|p_i^x - p_i^y\|_1 < \ve
\]
(where $\forall^\infty i$ means ``for all but finitely many $i$'')
and such that for every $i \in \N$,
the function $E \to \R$ defined by
$(x, y) \mapsto p_i^x(\{y\})$ is Borel.

A CBER $E$ on $X$ is \textbf{Borel $2$-amenable}
if there is a sequence $(p_{i, j})_{i, j \in \N}$,
where each $p_{i, j}$ is a function taking every $x \in X$
to some $p_{i, j}^x \in \Prob([x]_E)$,
such that
\[
    \forall (x, y) \in E\; 
    \forall \ve > 0 \;
    \forall^\infty i \;
    \forall^\infty j \;
    \|p_{i, j}^x - p_{i, j}^y\|_1 < \ve,
\]
and such that for every $(i, j) \in \N^2$,
the function $E \to \R$ defined by
$(x, y) \mapsto p_{i, j}^x(\{y\})$ is Borel.

Let $\Phi \in \{\text{hyperfinite}, \text{amenable}, \text{$2$-amenable}\}$,
and let $E$ be a CBER on a standard Borel space $X$.
Given a Borel probability measure $\mu$ on $X$,
we say that $E$ is \textbf{$\mu$-$\Phi$}
if there is a $\mu$-conull subset $X' \subseteq X$ 
such that $E \uhr X'$ is ``Borel $\Phi$''.
We say that $E$ is \textbf{measure-$\Phi$}
if $E$ is $\mu$-$\Phi$
for every Borel probability measure $\mu$ on $X$.
If we require $X'$ to be invariant,
we get an a priori stronger condition for being $\mu$-$\Phi$ for a single $\mu$.
However, if we allow ourselves to pass to a stricter measure,
e.g. $\sum \frac{1}{2^n} g_n \mu$
where $G=(g_n)_n$ is a group whose action generates $E$,
the $E$-saturation of a null set will remain a null set,
and hence we could assume $X'$ to be $E$-invariant.

The following is essentially in the proof of \cite[2.13(ii)]{JKL02},
see also \cite[Theorem 9.21]{Kec24}.

\begin{prop}\label{CFW}
    Let $E$ be a CBER on a standard Borel space $X$.
    If $\mu$ is a Borel probability measure on $X$
    and $E$ is $\mu$-$2$-amenable,
    then $E$ is $\mu$-hyperfinite.
    In particular,
    if $E$ is Borel $2$-amenable,
    then $E$ is measure-hyperfinite.
\end{prop}
\begin{proof}
    By the Connes-Feldman-Weiss theorem \cite{CFW81},
    it suffices to show that $E$ is $\mu$-amenable.
    By the Feldman-Moore theorem \cite{FM77},
    there is a countable group $G$
    and a Borel action $G \car X$ such that $E = E_G^X$.
    Since $E$ is $\mu$-$2$-amenable,
    we have
    \begin{equation}\label{eq:2amenable}
        \forall^\mu x \;
        \forall g \in G \;
        \forall \ve > 0 \;
        \forall^\infty i \;
        \forall^\infty j \;
        \|p_{i, j}^x - p_{i, j}^{gx}\|_1 < \ve.
    \end{equation}
    
    We claim that there exists a function $f : \N \to \N$ such that
    \begin{equation}\label{eq:f}
        \forall^\mu x \;
        \forall g \in G \;
        \forall \ve > 0 \;
        \forall^\infty i \;
        \forall j \ge f(i) \;
        \|p_{i, j}^x - p_{i, j}^{gx}\|_1 < \ve.
    \end{equation}
    Indeed, let $G=\{g_1,g_2,\cdots\}$ be an enumeration of the elements of $G$ and define a Borel probability measure $\nu$ on $X\times G \times \N$ satisfying $\nu|_{X\times\{(g_n,m)\}}=\frac{1}{2^{nm}}\mu$ for any $(n,m)\in\N^2$. For $i,j\in\N$, we define
    \[
    A_{ij}=\{(x,g,m)\in X\times G \times \N \mid \forall j'\ge j, ~ \|p_{i, j'}^x - p_{i, j'}^{gx}\|_1 < \frac{1}{m} \}.
    \]
    Since $(A_{ij})_{j\in\N}$ is increasing and $\nu(X\times G \times \N)<\infty$, for each $i\in\N$, there exists $f(i)\in\N$ such that the set $A_i=\bigcup_{j\in\N}A_{ij} \setminus A_{if(i)}$ satisfies $\nu(A_i)<\frac{1}{2^i}$. By the Borel-Cantelli lemma, the set $A=\bigcap_{k\in\N}\bigcup_{i\ge k}A_i$ is $\nu$-null. Define $A_{(g,m)}=\{x\in X \mid (x,g,m)\in A \}$ for each $(g,m)\in G\times\N$ and $B=\bigcup_{(g,m)\in G\times \N}A_{(g,m)}$. Note that for $x\notin B$, we have by definition that for each $(g,m)\in G\times \N$, $(x,g,m)\notin A$, i.e. $(x,g,m)\notin A_{i}$ for all but finitely many $i$. In other words, together with \eqref{eq:2amenable} this implies
    \begin{equation}\label{eq:B}
        \forall^\mu x \in X\setminus B \;
        \forall g \in G \;
        \forall m \in \N \;
        \forall^\infty i \;
        \forall j \ge f(i) \;
        \|p_{i, j}^x - p_{i, j}^{gx}\|_1 < \frac{1}{m}.
    \end{equation}
    Meanwhile, $B$ is $\mu$-null since $\mu(A_{(g_n,m)})=2^{nm} \nu(A_{(g_n,m)})\le 2^{nm} \nu(A) = 0$. Hence, \eqref{eq:B} implies \eqref{eq:f}. By \eqref{eq:f}, $(p_{i, f(i)})_{i \in \N}$ witnesses $\mu$-amenability of $E$.
\end{proof}

\begin{defn}
    Let $G \car X$ be a group action.
    Given a finite subset $S \subset G$ and $\ve > 0$,
    an \textbf{$(S, \ve)$-Reiter function}
    is a probability measure $p$ on $G$
    such that for every $s \in S$ and every $x \in X$,
    we have $\|px - psx\|_1 < \ve$ (where $py\in \Prob(G\cdot y)$ denotes the pushforward measure under the map $g\mapsto gy$).
\end{defn}

\begin{defn}\label{uca}
    Let $G$ be a group.
    A subgroup $H \le G$ is \textbf{uniformly coamenable}
    if the action $G \car G/H$ satisfies that
    for every finite $S \subset G$ and every $\ve > 0$,
    there is an $(S, \ve)$-Reiter function.
\end{defn}

Natural examples of uniformly coamenable subgroups include coamenable normal subgroups. 

\begin{lem}\label{uniformlyCoamenable}
    Let $G \car X$ be a Borel action
    of a countable group on a standard Borel space.
    If every stabilizer is uniformly coamenable in $G$,
    then $E_G^X$ is Borel amenable.
\end{lem}

\begin{proof}
    Denote by $\Prob_*(G)$ the subset of $\Prob(G)$
    consisting of finitely supported $\Q$-valued probability measures.
    Fix an ordering on $\Prob_*(G)$ isomorphic to $(\N, <)$,
    and fix an increasing sequence $(S_n)_{n \in \N}$
    of finite subsets of $G$ such that $\bigcup_n S_n = G$.

    For every $n \in \N$ and $x \in X$,
    let $p_n^x$ be the first (in the ordering)
    element of $\Prob_*(G)$
    which is an $(S_n, 1/n)$-Reiter function for $G \car G/\Stab(x)$;
    this exists since $\Prob_*(G)$ is dense in $\Prob(G)$.
    Set $q_n^x = p_n^x x$.

    Now fix $x \in X$ and $g \in G$.
    Since $p_n^x = p_n^{gx}$ and $\|p_n^x x - p_n^x gx\|<\frac{1}{n}$ if $g\in S_n$,
    we have
    \[
        \|q_n^x - q_n^{gx}\|
        = \|p_n^x x - p_n^{gx} gx\|
        = \|p_n^x x - p_n^x gx\|
        \to 0.
    \]
    Thus $(q_n)_{n \in \N}$ witnesses Borel amenability of $G \car X$.
\end{proof}

\begin{lem}\label{amenableUnion}
    Let $E$ be a CBER on a standard Borel space $X$,
    and let $(A_n)_{n \in \N}$ be a cover of $X$ by Borel sets.
    If each $E \uhr A_n$ is Borel amenable,
    then so is $E$.
\end{lem}
\begin{proof}
    By \cite[Proposition 2.15(iv)]{JKL02},
    each $E \uhr [A_n]_E$ is Borel amenable,
    so by replacing each $A_n$ with $[A_n]_E$,
    we can assume that each $A_n$ is $E$-invariant.
    By \cite[Proposition 2.15(i)]{JKL02},
    by passing to subsets,
    we can further assume that $(A_n)_{n \in \N}$ is a partition.
    Then we are done by \cite[Proposition 2.15(vi)]{JKL02}.
\end{proof}

We are now ready to prove \Cref{mthm:B} from the introduction.

\begin{thm}\label{amenableTheorem}
    Let $G \car T$ be an action of a countable group on a countable tree
    such that every geodesic ray $\mathbf v$
    has an initial segment $\sigma$ such that
    $\Stab(\mathbf v)$ is uniformly coamenable in $\Stab(\sigma)$.
    Then $E_G^{\partial T}$ is Borel $2$-amenable,
    so in particular measure-hyperfinite.
\end{thm}

\begin{proof}
    For every geodesic $\sigma$,
    let $\Geo_\sigma(T)$ be the set of geodesic rays $\mathbf v$ extending $\sigma$
    such that $\Stab(\mathbf v)$ is uniformly coamenable in $\Stab(\sigma)$.
    Then $(\Geo_\sigma(T))_\sigma$ is a cover of $\Geo(T)$. We also claim that $\Geo_\sigma(T)$ is Borel for any geodesic $\sigma$. Indeed, let $(S_n)_{n\in\N}$ and $\Prob_*(G)$ be as in the proof of \Cref{uniformlyCoamenable} and for each $n\in\N$ and $p \in \Prob_*(G)$, we define
    \[
    A_{n,p}=\{\mathbf{v} \in \Geo(T) \mid \sigma \prec \mathbf{v} \wedge \max_{g\in S_n}\|p\mathbf{v}-pg\mathbf{v}\|_1 \leq \frac{1}{n} \}.
    \]
    Since the action $G\car\Geo(T)$ is an action by homeomorphisms, $A_{n,p}$ is closed. Hence,  $\Geo_\sigma(T)$ is Borel since $\mathbf{v} \in \Geo_\sigma(T) \Leftrightarrow [\forall n\in \N~ \exists p\in \Prob_*(G)~ \mathbf{v} \in A_{n,p}]$. For every $\sigma$,
    we have
    $E_G^{\Geo(T)} \uhr \Geo_\sigma(T)
    = E_{\Stab(\sigma)}^{\Geo_\sigma(T)}$,
   and every stabilizer of the action $\Stab(\sigma)\car \Geo_\sigma(T)$ is uniformly coamenable in $\Stab(\sigma)$ by hypothesis, 
    so this equivalence relation is Borel amenable by \Cref{uniformlyCoamenable}.
    Hence $E_G^{\Geo(T)}$ is Borel amenable by \Cref{amenableUnion}, and 
    thus $E_{G \times \N}^{\Geo(T)}$ is Borel $2$-amenable
    by \cite[Proposition 2.15(ix)]{JKL02}.
    By \Cref{quotient},
    we have $E_G^{\partial T} \sim_B E_{G \times \N}^{\Geo(T)}$.
    Since Borel $2$-amenability is closed under Borel reduction
    (see \cite[Proposition 2.15(ii)]{JKL02}),
    we get that $E_G^{\partial T}$ is Borel $2$-amenable.
    Measure-hyperfiniteness follows from \Cref{CFW}.
\end{proof}

\section{Prodiscrete actions}\label{prodiscrete}
A Borel action of a countable group $G$
on a standard Borel space is \textbf{prodiscrete}
(or \textbf{modular})
if it is the inverse limit of countable $G$-sets,
or equivalently,
if it is isomorphic to
the boundary action $G \car \partial T$
of a rooted action $G \car T$ on a countable tree,
where an action is \textbf{rooted} if it has a fixed point
(see \cite[Fact 1.1]{Kec05}).
See \cite[Section 10.8]{Kec24} for background on prodiscrete actions.

\begin{eg}
    Let $G$ be a countable residually finite group.
    Then $G$ admits a free pmp prodiscrete action
    (see \cite[Fact 1.4]{Kec05}),
    meaning that there is a rooted action $G \car T$ on a countable tree
    and a Borel probability measure $\mu$ on $\partial T$,
    such that the boundary action $G \car \partial T$
    preserves $\mu$ and is $\mu$-almost everywhere free.
    \begin{itemize}
        \item
            If $G$ is non-amenable
            (for instance,
            if $G$ is a non-abelian free group),
            then $E_G^{\partial T}$ is not $\mu$-hyperfinite
            (see \cite[Proposition 2.5(ii)]{JKL02}).
        \item
            If $G$ is antitreeable,
            then $E_G^{\partial T}$ is not $\mu$-treeable.
            A countable group is \textbf{antitreeable}
            if it admits no treeable free pmp actions
            (see \cite[Section 10.7(2)]{Kec24}).
            There are many examples of antitreeable groups,
            for instance,
            any infinite property (T) group,
            or any product of a non-amenable group by an infinite group.
        \end{itemize}
\end{eg}
It is known that prodiscrete actions cannot be arbitrarily complicated,
since for instance,
it was shown in \cite{Hjo12}
that the free part of the Bernoulli shift $F_2 \car 2^{F_2}$
does not Borel reduce to the orbit equivalence relation
of any prodiscrete action
(see \cite[Section 10.8]{Kec24} for more examples,
where they are called ``antimodular'').
It is not known whether this result can be extended to all boundary actions:
\begin{qn}
    If $G \car T$ is an action of a countable group on a countable tree,
    is it true that the free part of the Bernoulli shift $F_2 \car 2^{F_2}$
    does not Borel reduce to $E_G^{\partial T}$?
\end{qn}

\printbibliography
\end{document}